\documentclass[a4paper,12pt]{article}

    \usepackage{amsmath}
    \usepackage{amssymb}
    \usepackage{amsthm}

\usepackage{subfigure}

\usepackage{pstricks}
\usepackage{graphicx,psfrag}

\newcommand {\al}   {\alpha}          \newcommand {\bt}  {\beta}
\newcommand {\gam } {\gamma}

     \newcommand {\pl}   {\partial}
        
\newcommand {\RRR}  {{\mathbb R}}     
               
\newcommand {\EEE}  {{\cal E}}        \newcommand {\HHH}  {{\cal H}}

 \newcommand {\beq}  {\begin{equation}} \newcommand {\eeq}  {\end{equation}}
\newcommand {\beqo}  {\begin{equation*}}  \newcommand {\eeqo}  {\end{equation*}}

      \newtheorem{theorem}{Theorem}
      \newtheorem{lemma}{Lemma}
      
      \newtheorem{zam}{Remark}
      \newtheorem{opr}{Definition}

\author{Alexander Plakhov\thanks{Department of Mathematics, University of Aveiro, Portugal
and Institute for Information Transmission Problems, Russia} \and Vera Roshchina\thanks{CIMA, University of \'{E}vora,  Portugal; Ci\^{e}ncia 2008}}

\title{Bodies invisible from one point}

\date{}

\begin{document}

\maketitle

\begin{abstract}

We show that there exist bodies with mirror surface invisible from
a point in the framework of geometrical optics. In particular,
we provide an example of a connected three-dimensional body
invisible from one point.
\end{abstract}

\begin{quote}
\end{quote}

\begin{quote}
\end{quote}

The issue of invisibility attracts a lot of attention nowadays.
Various physical and technological ideas aiming at creation of
objects invisible for light rays are being widely discussed. A
brief review of the major developments is provided in our
recent work \cite{PR invisibility}. For a more entertaining
treatment of the subject the reader may refer to a recent
article in the BBC Focus Magazine \cite{BBC-Focus}. Most of the
recent developments are based on the wave representation of
light (e.g. \cite{ErginEtAl, ShurigEtAl, ValentineEtAl}).
However, here we study the notion of invisibility from the
viewpoint of geometrical optics. In other words, we consider a
model where a bounded open set $B$ with a piecewise smooth
boundary in Euclidean space $\RRR^d$, $d\geq 2$ represents a
physical body with mirror surface, and the billiard in the
complement of this domain, $\RRR^d \setminus B$, represents
propagation of light outside the body. This work continues the
series of results on invisibility obtained in \cite{0-resist,PR
invisibility,PR fractal}.

A semi-infinite broken line $l \subset \RRR^d \setminus B$ with
the endpoint at $O \in \RRR^d \setminus \bar B$ is called a
billiard trajectory emanating from $O$, if the endpoints of its
segments (except for $O$) are regular points of the body
boundary $\pl B$ and the outer normal to $\pl B$ at any such
point is the bisector of the angle formed by the segments
adjoining the point.

\begin{opr}\label{o invis 1 point}
\rm We say that $B$ is {\it invisible from a point} $O \not\in
\bar B$,  if for almost any ray with the vertex at $O$ there
exists a billiard trajectory emanating from $O$ with a finite
number of segments such that the first segment (adjoining $O$)
and the last (infinite) segment belong to the ray (see Fig.
\ref{fig:traj}).
\end{opr}

\begin{figure}[h]
\begin{picture}(0,90)
\rput(4,0.8){
\scalebox{1}{

\psline[linecolor=red,arrows=->,arrowscale=1.5](0,0)(2,0.5)
\psline[linecolor=red,arrows=->,arrowscale=1.5](2,0.5)(3.5,-0.5)(5,0)
\psline[linecolor=red,arrows=->,arrowscale=1.5](5,0)(5,1.25)(8,2)
\psline[linecolor=blue,linestyle=dashed](2,0.5)(5,1.25)
\psdots(0,0)
\rput(-0.35,-0.15){$O$}
}}
\end{picture}
\caption{A billiard trajectory emanating from $O$ with the first and last segments contained in a ray with vertex at $O$ (the body itself is not shown).
}
\label{fig:traj}
\end{figure}
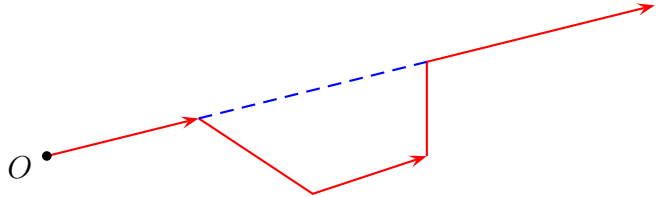

The main result of this note is the following

\begin{theorem}\label{t invis1point}
For each $d$ there exists a body $B \subset \RRR^d$ invisible
from a point. If $d \ge 3$ then the body is connected.
\end{theorem}

Not much is known today about invisibility in the billiard
setting. In the limit where the point of reference $O$ goes to
infinity the notion of invisibility from a point is transformed
into the notion of invisibility in one direction (see
\cite{0-resist}, \cite{PR invisibility} and \cite{PR fractal}).
There exist two- and three-dimensional bodies invisible in one
direction \cite{0-resist} and three-dimensional bodies
invisible in two orthogonal directions \cite{PR invisibility}.
It is straightforward to generalize these results to obtain
$d$-dimensional bodies that are invisible in $d-1$ mutually
orthogonal directions. On the other hand, there are no bodies
invisible in {\it all} directions (or, equivalently, invisible
from {\it all} points).

The proof of the theorem is based on a direct construction. In the proof we use the following lemma.

\begin{lemma}[A characteristic property of a bisector in a triangle]\label{l charsvvo} Consider a triangle $ABC$ and  a point $D$ lying on the side
$AC$. Let $AB = a_1$,\, $BC = a_2$,\, $AD = b_1$,\, $DC = b_2$,
and $BD = f$ (see Fig. \ref{fig bisector}).
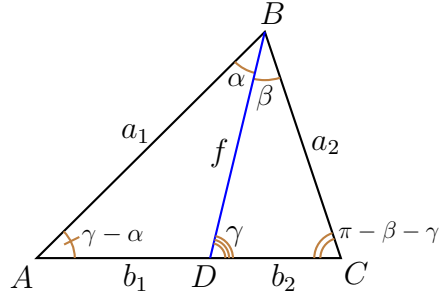
\begin{figure}[h]
\begin{picture}(0,120)
\rput(1.5,1.5){
\scalebox{1}{
\rput(4,-1){
\pspolygon(0,0)(4,0)(3,3)
\psarc[linecolor=brown](3,3){0.56}{225.3}{256.5}
\psarc[linecolor=brown](3,3){0.64}{256.7}{288.2}
\rput(2.63,2.36){\scalebox{0.9}{$\al$}}
\rput(3,2.13){\scalebox{0.9}{$\bt$}}
\psline[linewidth=0.8pt,linecolor=blue](3,3)(2.28,0)
\rput(-0.2,-0.25){$A$}
\rput(3.1,3.25){$B$}
\rput(4.2,-0.2){$C$}
\rput(2.2,-0.23){$D$}
\psarc[linecolor=brown](2.28,0){0.3}{0}{76.2}
\psarc[linecolor=brown](2.28,0){0.25}{0}{76.2}
\psarc[linecolor=brown](2.28,0){0.2}{0}{76.2}
\rput(2.6,0.3){$\gam$}
\rput(1.3,1.65){$a_1$}
\rput(3.8,1.5){$a_2$}
\rput(1.3,-0.25){$b_1$}
\rput(3.25,-0.25){$b_2$}
\rput(2.4,1.4){$f$}
\psarc[linecolor=brown](4,0){0.35}{108.5}{180}
\psarc[linecolor=brown](4,0){0.27}{108.5}{180}
\psarc[linecolor=brown](0,0){0.5}{0}{45}
\psline[linecolor=brown](0.36,0.18)(0.56,0.28)
\rput(1,0.3){\scalebox{0.8}{$\gam-\al$}}
\rput(4.63,0.35){\scalebox{0.75}{$\pi-\bt-\gam$}}
}
}}
\end{picture}
\caption{A characteristic property of a bisector in a triangle.
}
\label{fig bisector}
\end{figure}
The segment $BD$ is the bisector of the angle $\measuredangle ABC$ if and only if $$(a_1 + b_1)(a_2 - b_2) = f^2.$$
\end{lemma}

\begin{proof}
Consider the following relations on the values $a_1$,\, $a_2$,\, $b_1$,\, $b_2$, and $f$:

(a)\ \, $a_1/a_2 = b_1/b_2$;

(b)\ \, $a_1 a_2 - b_1 b_2 = f^2$;

(c)\ \, $(a_1 + b_1)(a_2 - b_2) = f^2$.

The equalities (a) and (b) are well known; each of them is a characteristic property of triangle bisector as well. It is interesting to note that each of these algebraic relations is a direct consequence of the two others.

Assume that $BD$ is the bisector of the angle $\measuredangle ABC$. Then the equalities (a) and (b) are true, therefore (c) is also true. The direct statement of the lemma is thus proved.

To derive the inverse statement, we need to apply the sine rule and some trigonometry. Denote $\al = \measuredangle ABD$,\, $\bt = \measuredangle CBD$, and $\gam = \measuredangle BDC$ (see Fig. \ref{fig bisector}). Applying the sine rule to $\triangle ABD$, we have
$$
\frac{a_1}{\sin\gam} = \frac{b_1}{\sin\al} = \frac{f}{\sin(\gam-\al)},
$$
and applying the sine rule to $\triangle BDC$, we have
$$
\frac{a_2}{\sin\gam} = \frac{b_2}{\sin\bt} = \frac{f}{\sin(\gam+\bt)}.
$$
This implies that
$$
a_1 + b_1 = \frac{f}{\sin(\gam-\al)}\, (\sin\gam + \sin\al) = f\, \frac{\sin\frac{\gam+\al}{2}}{\sin\frac{\gam-\al}{2}},
$$
$$
a_2 - b_2 = \frac{f}{\sin(\gam+\bt)}\, (\sin\gam - \sin\bt) = f\, \frac{\sin\frac{\gam-\bt}{2}}{\sin\frac{\gam+\bt}{2}}.
$$
Using the equality (c), one gets
$$
f^2\ \frac{\sin\frac{\gam+\al}{2}\, \sin\frac{\gam-\bt}{2}}{\sin\frac{\gam-\al}{2}\, \sin\frac{\gam+\bt}{2}}\, =\, f^2,
$$
whence
$$
\sin\frac{\gam+\al}{2}\, \sin\frac{\gam-\bt}{2}\, =\, \sin\frac{\gam-\al}{2}\, \sin\frac{\gam+\bt}{2},
$$
After some algebra as a result we have
$$
\cos \Big(\gam + \frac{\al-\bt}{2} \Big)\, =\, \cos \Big(\gam - \frac{\al-\bt}{2}\Big).
$$
The last equation and the conditions $0 < \al, \ \bt, \ \gam < \pi$ imply that $\al = \bt$. The inverse statement of the lemma is also proved.
\end{proof}

\begin{proof} \rput(-0.3,-0.03){\psframe[linecolor=white,fillstyle=solid,fillcolor=white](0,0)(0.1,0.1)} \!\!{\it of Theorem \ref{t invis1point}.}
Consider confocal ellipse and hyperbola on the plane. In a
convenient coordinate system in which the major and minor axes
of the ellipse coincide with the coordinate axes, the ellipse
is given by the equation
$$
\frac{x^2}{a^2} + \frac{y^2}{b^2} = 1, \qquad a>b>0,
$$
and the hyperbola is given by the equation
$$
\frac{x^2}{\al^2} - \frac{y^2}{\bt^2} = 1, \qquad
$$
with the relation \beq\label{usl-e0} c^2 = a^2 - b^2 = \al^2 +
\bt^2 \eeq ensuring that the ellipse and the parabola are
confocal. Observe that the two foci $F_1$ and $F_2$ are located
at $(\pm c, 0)$. Finally, we require the intersection of the
ellipse with each of the branches of the hyperbola to be
in-line with the relevant focus (the dashed line on
Fig.~\ref{fig simp ellipse and parabola}).
\begin{figure}[h]
\centering \includegraphics[width=200pt, keepaspectratio ]{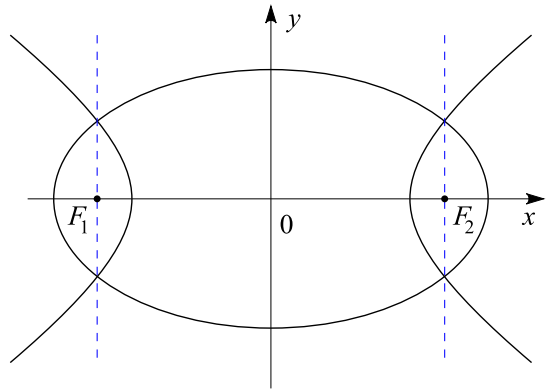}
\caption{An ellipse and hyperbola satisfying the conditions
\eqref{usl-e0} and \eqref{usl-e}.} \label{fig simp ellipse and parabola}
\end{figure}
It is an elementary exercise to check that this property is
guaranteed by the condition
 \beq\label{usl-e}
\frac{1}{\bt^2} - \frac{1}{b^2} = \frac{1}{c^2}. \eeq

On Fig. \ref{fig ellipse and parabola} the ellipse is indicated
by $\EEE$, the right branch of the hyperbola by $\HHH$ (the
other branch is not considered), and the foci by $F_1$ and
$F_2$.
\begin{figure}[h]
\centering \includegraphics[width=200pt, keepaspectratio ]{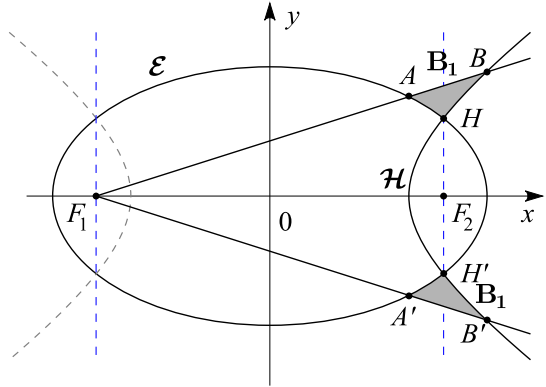}
\caption{A body having zero resistance to a flow of particles emanating from a point.}
\label{fig ellipse and parabola}
\end{figure}

Choose an arbitrary point $B$ on $\HHH$, such that it lies
outside of the ellipse, and denote by $A$ the intersection of
the segment $F_1B$ with the ellipse. Let $H$ be the closest
intersection point with the ellipse as we move from $B$ along
the branch of the hyperbola towards the ellipse  (see
Fig.~\ref{fig ellipse and parabola}, where the relevant point
$B$ happens to belong to the top segment of the hyperbola's
branch $\HHH$). We are interested in the curvilinear triangle
$ABH$ and its symmetric (w.r.t. the major axis of the ellipse)
counterpart $A'B'H'$. We denote the union of the aforementioned
curvilinear triangles by $\mathbf{B_1}$ (shown in grey color on
Fig.~\ref{fig ellipse and parabola}). By construction
\beq\label{85e1} \measuredangle AF_1F_2 = \measuredangle
A'F_1F_2. \eeq

Consider a particle emanating from $F_1$ and making a
reflection from $\mathbf{B_1}$. The first reflection is from
one of the arcs $AH$ or $A'H'$. Without loss of generality the
point of first reflection $C$ lies on $A'H'$ (see Fig.~\ref{fig
ellipse and parabola particle}).
\begin{figure}[h]
\centering \includegraphics[width=200pt, keepaspectratio ]{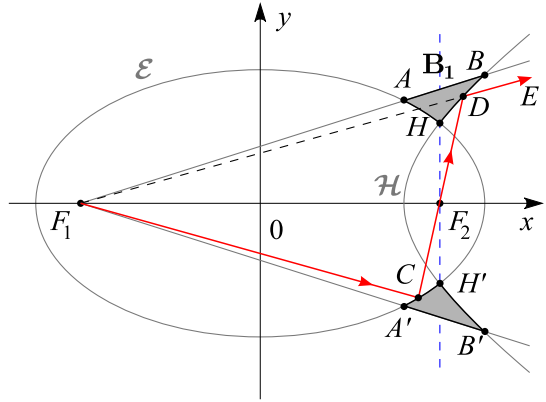}
\caption{The trajectory of a particle emanating from $F_1$.}
\label{fig ellipse and parabola particle}
\end{figure}
 We have
\beq\label{85e2} \measuredangle CF_1F_2 < \measuredangle
A'F_1F_2. \eeq After the reflection the particle passes through
the focus $F_2$ and then intersects $\HHH$ at a point $D$
(recall that by construction the segments $HF_2$ and $H'F_2$
are orthogonal to the major axis $F_1F_2$, hence, the
intersection of the ray that emerges from $C$ and passes
through $F_2$ necessarily reaches the hyperbola outside of the
ellipse).

By the focal property of ellipse we have \beq\label{85e3}
|F_1C| + |F_2C| = |F_1H| + |F_2H|, \eeq and by the focal
property of hyperbola, \beq\label{85e4} |F_1D| - |F_2D| =
|F_1H| - |F_2H|. \eeq Multiplying both parts of (\ref{85e3})
and (\ref{85e4}) and bearing in mind that $F_1F_2$ is
orthogonal to $F_2H$, we obtain \beq\label{85e5} (|F_1C| +
|F_2C|)(|F_1D| - |F_2D|) = |F_1H|^2 - |F_2H|^2 = |F_1F_2|^2.
\eeq Applying Lemma \ref{l charsvvo} to the triangle $CF_1D$
and using (\ref{85e5}) we conclude that $F_1F_2$ is a bisector
of this triangle, that is, \beq\label{85e6} \measuredangle
CF_1F_2 = \measuredangle DF_1F_2. \eeq Using (\ref{85e1}),
(\ref{85e2}), and (\ref{85e5}), we obtain that $\measuredangle
DF_1F_2 < \measuredangle AF_1F_2$, therefore $D$ lies on the
arc $HB$. After reflecting at $D$ the particle moves along the
line $DE$ containing $F_1$. This property can be interpreted as
 $B_1$ having zero resistance to the flow of particles
emanating from $F_1$.

Now consider the body $\mathbf{B_2}$ obtained  from
$\mathbf{B_1}$ by dilation with the center at $F_1$ and such
that $\mathbf{B_1}$ and $\mathbf{B_2}$ have exactly two points
in common (in Fig. \ref{fig invisible from 1 point} the
dilation coefficient is greater than 1). A particle emanating
from $F_1$ and reflected from $\mathbf{B_1}$ at $C$ and $D$,
further moves along the line $DE$ containing $F_1$, besides the
equality (\ref{85e6}) takes place.
\begin{figure}[h]
\centering \includegraphics[width=260pt, keepaspectratio ]{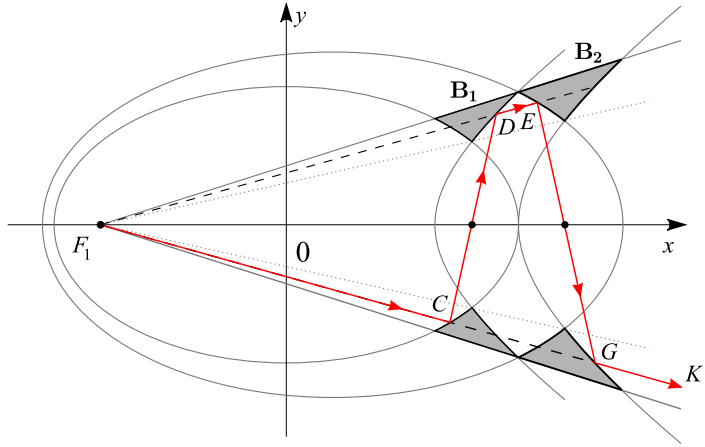}
\caption{A body invisible from one point.} \label{fig invisible
from 1 point}
\end{figure}

Then the particle makes two reflections from $\mathbf{B_2}$ at
$E$ and $G$ and moves freely afterwards along a line containing
$F_1$, besides the equality \beq\label{85e7} \measuredangle
EF_1F_2 = \measuredangle GF_1F_2. \eeq takes place. Using
(\ref{85e6}) and (\ref{85e7}), as well as the (trivial)
equality $\measuredangle DF_1F_2 = \measuredangle EF_1F_2$, we
find that \beqo\label{85e8} \measuredangle CF_1F_2 =
\measuredangle GF_1F_2. \eeqo This means that the initial
segment $F_1C$ of the trajectory and its final ray $GK$ lie in
the same ray $F_1K$. The rest of the trajectory, the broken
line $CDEG$, belongs to the convex hull of the set
$\mathbf{B_1} \cup \mathbf{B_2}$. Thus we have proved that
$\mathbf{B_1} \cup \mathbf{B_2}$ is a two-dimensional body
invisible from the point $F_1$.

In the case of a higher dimension $d$ the  (connected) body
invisible from $F_1$ is obtained by rotation of $\mathbf{B_1}
\cup \mathbf{B_2}$ about the axis $F_1 F_2$: a
three-dimensional body is shown on Fig.~\ref{fig first
rotation}. Observe that because of the rotational symmetry of
the body the trajectory of the particle emitted from the
relevant focal point (that corresponds to $F_1$ in the
two-dimensional case) lies within a plane that contains the
major axis of the relevant ellipsoids.
\begin{figure}[h]
\centering \includegraphics[height=261pt, keepaspectratio ]{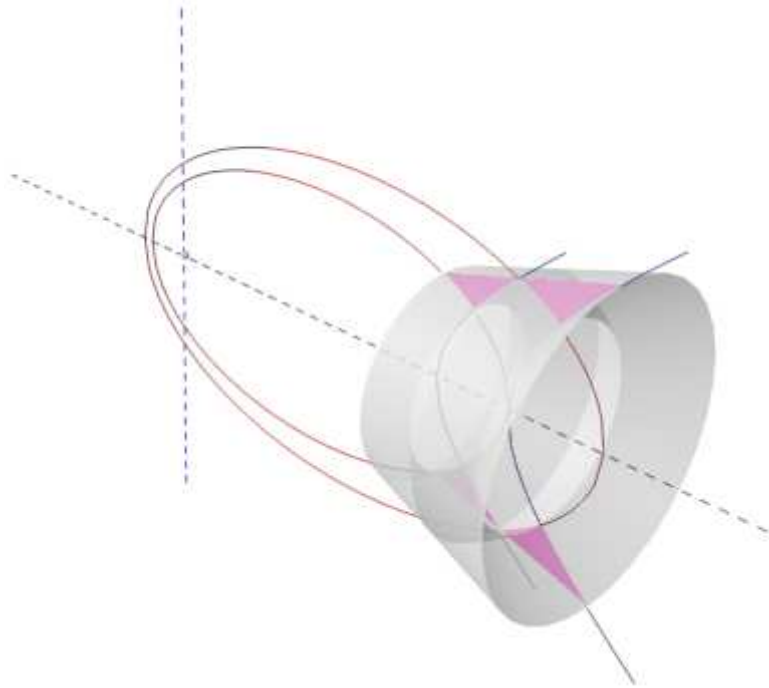}
\caption{A connected body invisible from one point.} \label{fig first rotation}
\end{figure}

Another example of a three-dimensional body invisible from a
point can be obtained by rotating the two-dimensional
construction around the axis perpendicular to the major axes of
the ellipses and passing through the focal point $F_1$ (see
Fig.~\ref{fig second rotation}).
\begin{figure}[h]
\centering \includegraphics[height=211pt, keepaspectratio ]{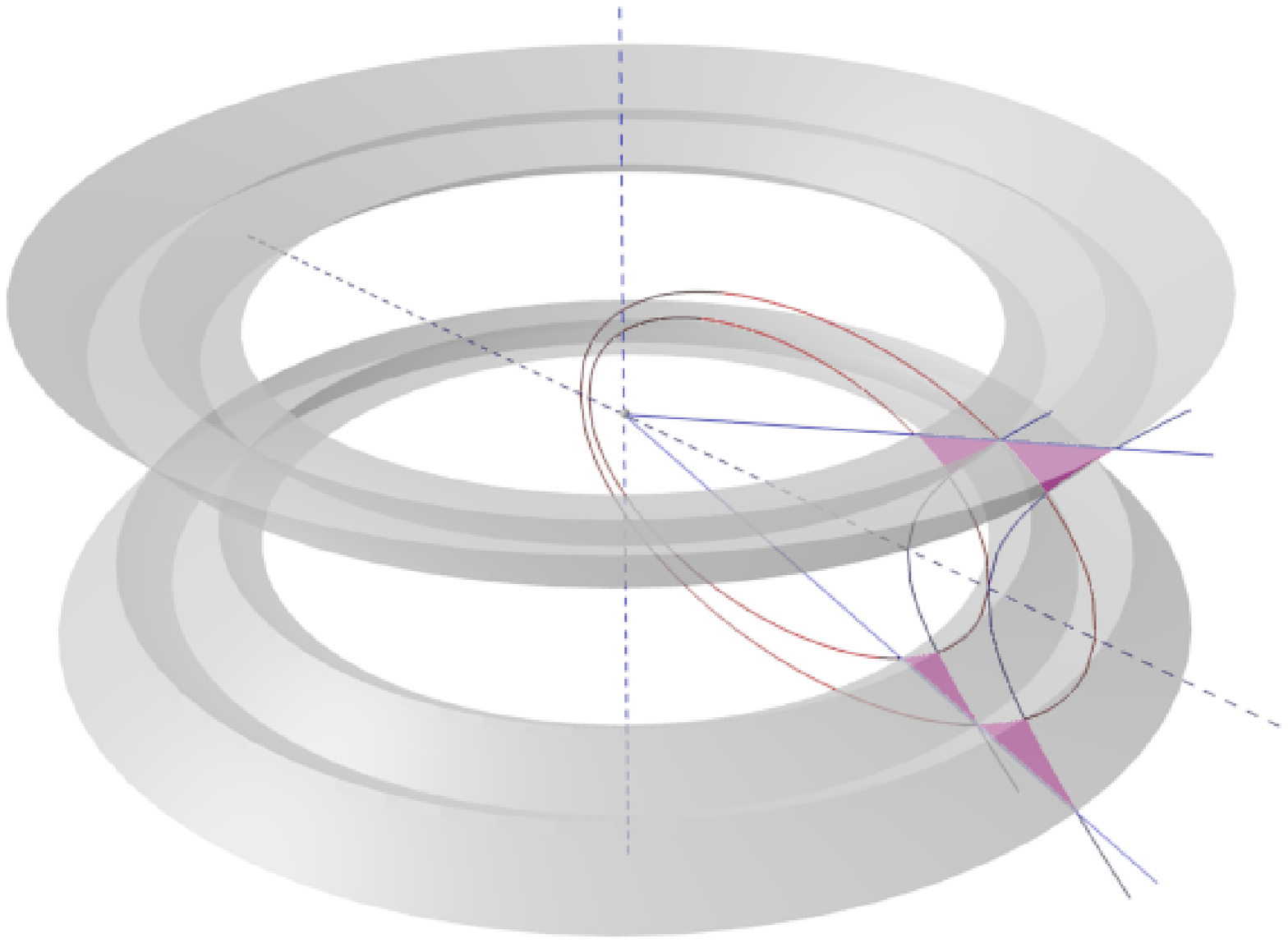}
\caption{Another body invisible from one point} \label{fig second rotation}
\end{figure}

\end{proof}

\begin{zam}\label{z 85}\rm
From the proof of the theorem we see that the invisible body is determined by 5 parameters: $a,\, b,\, \al,\, \bt$, and the inclination of the line $F_1B$, with 2 conditions imposed by (\ref{usl-e0}) and (\ref{usl-e}). Thus, the construction is defined by three parameters. One of them is the scale, and the second and third ones can be taken to be the angles $\measuredangle HF_1F_2$ and $\measuredangle BF_1F_2$.
\end{zam}

\section*{Acknowledgements}

This work was partly supported by {\it FEDER} funds through {\it COMPETE}--Operational Programme Factors of Competitiveness and by Portuguese funds through the {\it Center for Research and Development in Mathematics and Applications} (CIDMA) and the Portuguese Foundation for Science and Technology (FCT), within project PEst-C/MAT/UI4106/2011 with COMPETE number FCOMP-01-0124-FEDER-022690 and project PTDC/MAT/113470/2009.


\begin{thebibliography}{99}

\bibitem{BBC-Focus} {\it How to make anything invisible} BBC
    Focus, Issue 336, December 2011, 33--39.

\bibitem{0-resist} A. Aleksenko and A. Plakhov. {\it Bodies of
    zero resistance and  bodies invisible in one direction}.
    Nonlinearity {\bf 22}, 1247-1258 (2009).

\bibitem{ErginEtAl}T. Ergin, N. Stenger, P. Brenner, J.~B.~
    Pendry and M. Wegener. {\it Three-Dimensional Invisibility
    Cloak
    at Optical Wavelengths.} Science 328, 337–339 (2010).



\bibitem{PR invisibility} A. Plakhov and V. Roshchina. {\it
    Invisibility in billiards}. Nonlinearity {\bf 24}, 847--854
    (2011).

\bibitem{PR fractal} A. Plakhov and V. Roshchina. {\it
    Fractal bodies invisible in 2 and 3 directions}. arXiv:1107.5667 (2011).

\bibitem{ShurigEtAl} D.~Schurig, J.~J.~Mock, B.~J.~Justice, %
    S.~A.~Cummer, J.~B.~Pendry, A.~F.~Starr and D.~R.~Smith {\it
    Metamaterial Electromagnetic Cloak at Microwave
    Frequencies}, Science 314, 977 (2006).

\bibitem{ValentineEtAl} J.~Valentine, S.~Zhang, T.~Zentgraf,
    E.~Ulin-Avila, D.~A.~Genov, G.~Bartal and X.~Zhang. {\it Three-dimensional
    optical metamaterial with a negative refractive index.}
    Nature, 455 (2008).


\end{thebibliography}
\end{document}